\documentclass[12pt]{article}
\usepackage{amssymb,amsmath,amsthm}
\usepackage{latexsym}
\usepackage[OT4]{fontenc}
\usepackage[utf8]{inputenc}

\def\er{{\mathbb R}}
\def\E{{\mathbb E}}

\def\var{{\mathrm{Var}}}

\def\tf{\tilde{f}}
\def\R{{\er}}
\def\mi{{\mu}}
\def\supp{{\mathrm{supp}}}
\def\ra{{\rightarrow}}
\def\is#1#2{\left\langle #1 , #2 \right\rangle}
\def\cov{\mathrm{cov}}
\def\Cov{\cov}
\def\tr{\mathrm{tr}}

\def\Id{{\mathrm{Id}}}
\def\id{\Id}

\def\eps{\varepsilon}
\def\spann{{\mathrm{span}}}

\def\TODO#1{\marginpar{\bf \tt TO DO: #1}}

\def\1{{\mathbf{1}}}

\def\LL{\mathcal{L}}

\newtheorem{thm}{Theorem}[section]
\newtheorem{prop}[thm]{Proposition}
\newtheorem{lem}[thm]{Lemma}
\newtheorem{cor}[thm]{Corollary}

\newtheorem{fact}[thm]{Fact}
\newtheorem{Def}[thm]{Definition} %%\numberwithin{equation}{subsection}

\def\KK{{\mathcal{K}}}
\def\bKK{{\bar{\mathcal{K}}}}
\def\LL{{\mathcal{L}}}
\def\bLL{{\bar{\mathcal{L}}}}
\def\HEN{{CH(d)}}

\title{No return to convexity
}
\author{Jakub Onufry Wojtaszczyk}

\date{19.X.2009}

\begin{document}
 \baselineskip=17pt 
\maketitle

\begin{abstract}
In the paper we study closures of classes of log--concave measures under taking weak limits, linear transformations and tensor products. We consider what uniform measures on convex bodies can one obtain starting from some class $\KK$. In particular we prove that if one starts from one--dimensional log--concave measures, one obtains no non--trivial uniform mesures on convex bodies.
\end{abstract}

\section{Introduction and notation}

The developments in asymptotic convex geometry generally tend to abandon the study of uniform measures on convex bodies in the favour of studying log--concave measures. On one hand, this is a natural generalization --- it is a well--known fact that any log--concave measure is a weak limit of projections of uniform measures on convex bodies. On the other hand, it allows one to use a wide variety of tools hitherto unavailable. It also gives rise to a plethora of new examples (to give one --- it is now possible to have a non--trivial one--dimensional case, as there are various one--dimensional log--concave measures, while one--dimensional convex bodies were a rather trivial subject study). A large number of strong results concerning convex bodies has been proved by passing through the domain of log--concave measures in a significant way.

A natural way to proceed in quite a few cases when considering the log--concave measures is some sort of induction upon dimension. As there is a well--developed theory of independent log--concave random variables, it is easy to study the tensor products of measures. It also frequently turns out that properties considered are easily seen to be preserved under linear transformations and weak limits. One may give a number of examples of classes of measures closed under these three operations. If one restricts oneself to non--degenerate linear transformations, then measures with the isotropic constant $L_\mi$ bounded by some given $C$ form probably the most important class (see, for instance, \cite{gianno} or \cite{geniso} for an analysis of the isotropic constant problem). A more interesting class of transformations, where one is also allowed to use projections (and, generally, degenerate linear maps) preserves the infimum convolution inequality (see \cite{LW}), and the same set of operations was considered in the context of \cite{Paouris} by Grigoris Paouris; a number of other examples are available. In this paper we shall study the more general case (where arbitrary linear transformations are taken), of course the case of non--degenerate linear maps is contained in it. While in the case of uniform measures on convex bodies it is easy to see that not much new is going to be obtained by applying these operations (for instance, if we begin by taking the one--dimensional convex bodies, ie. intervals, we end up only with parallelotopes), it is not obvious whether passing through the log--concave measures will help. One can see, for instance, that even starting only with intervals (that is, uniform measures on intervals), but working in the class of log--concave measures, one will obtain a wider variety of one--dimensional log--concave measures --- for instance the gaussian measure as the limit case of projecting the uniform measure on the cube onto the line spanned by the vector $(1,1,\ldots,1)$. Thus one might be inclined to hope that by proving a property studied is preserved under the given elementary operations in the log--concave setting one will obtain new, non--trivial examples of convex bodies satisfying the property. In this paper I intend to show that this is basically not the case. For instance, starting with all one--dimensional log--concave measures, one ends with quite a number of log--concave measure (including non--product ones), but the only measures equidistributed on convex bodies one obtains are those equidistributed on parallelotopes.

The result is essentially negative --- it proves that this is not the direction to pursue when attempting to prove new properties for convex bodies via log--concave measures. Its value, as in the case of most negative results, lies mainly in guiding other mathematicians away from this approach, rather than in direct application. As the approach, however, is not obviously wrong, the result still seems valuable.

\subsection{Definitions and notation}
For two sets $A, B$ in $\R^n$ by $A+B$ we denote their {\em Minkowski sum}, i.e. $\{a+b:a\in A, b\in B\}$, while by $tA$ for a real number $t$ we denote $\{ta : a\in A\}$. A log--concave measure in $\R^n$ (where log--concave is short for logarithmically concave) is a measure satisfying $\mi(tA + (1-t)B) \geq \mi^t(A) \mi^{1-t}(B)$ for any $A,B\in \R^n$ and any $t \in (0,1)$. We will assume all our log--concave measures are probabilistic and are not concentrated on any lower--dimensional subspace (that is, if $H \subset \R^n$ is an affine subspace of lower dimension, then $\mi(H) = 0$). A celebrated result of Borell states that any log--concave measure satisfying the above conditions has a density $g$ with respect to the Lebesgue measure and $\log g$ is concave.

We say a measure $\mi$ is {\em isotropic} if $\int_{\R^n} x d\mi(x) = 0$ and $\int_{\R^n} \is{x}{t}^2 d\mi(x) = 1$ for any $t$ on the unit Euclidean sphere. It is easy to see that any measure not concentrated on a lower--dimensional subspace has an affine image which is isotropic.

We shall call a random variable isotropic, log--concave, etc. if it is distributed according to a law which is isotropic, log--concave, etc.

We say a class of measures is closed under products if for any $\mi_1, \mi_2 \in \KK$ we have $\mi_1 \otimes \mi_2 \in \KK$. We say $\KK$ is closed under linear transformations if for any linear map $T$ and any $\mi \in \KK$ we have $\mi \circ T^{-1} \in \KK$. Finally, we say $\KK$ is closed under weak limits if for any sequence $(\mi_n)$ of measures from $\KK$ if the weak limit $\mi$ of $(\mi_n)$ exists, then it belongs to $\KK$.

$c, C, c_1, c_2, \ldots$ will always denote universal constants, possibly different from line to line. $c(n)$ and $C(n)$ are constants dependent only on the dimension $n$.

\subsection{Acknowledgements}
I would first and foremost like to thank the supervisor of my PhD thesis, Rafał Latała, for his unending patience and enormous help (in this case consisting both of presenting the question to me and answering my questions, especially in probability theory). I would like to thank Grigoris Paouris, who posed this question in July and motivated me to work on it. I would like to thank Radek Adamczak for pointing out an important mistake in my reasoning. I would also liked to thank all the people who made my life somewhat easier in the last few months, in particular Łukasz and Iwona Degórski, Marcin and Olga Pilipczuk and Kasia Staniszewska. Finally I would like to thank Jakub Pochrybniak for his help with my technical \TeX--related problems.

\section{Tensorization does not help}
This section is devoted to the proof of Theorem \ref{nomulti2}, which basically states that when one considers the closure of a class of measures with respect to products, linear transformations and weak limits and requires the result to be a uniform measure on a convex body, then it is enough to perform these operations one--by--one --- the result is a product of weak limits of linear transformations.

\subsection{Log--concave preliminaries}

The following facts are well--known:
\begin{fact} \label{szacowanie_gestosci} Let $\mi$ be an isotropic log--concave measure in $\R^n$. Then $\mi$ has a density $g$, and there exist such constants $c(n), C(n)$ dependent on $n$, that $$c(n) \leq g(0) \leq \sup_{x\in \R^n} g(x) \leq C(n).$$ \end{fact}

The bound on $\sup g(x)$ is easily equivalent to the existence of a bound on the isotropic constant, which may easily be seen to be bounded by $\sqrt{n}$, see  \cite{gianno} or \cite{geniso} for information on this relation and \cite{Klartag} for the best currently known bound of $\sqrt[4]{n}$. The relation between $g(0)$ and $\sup g(x)$ is studied for instance by Matthieu Fradelizi in \cite{Fradelizi}.

\begin{fact} \label{ciecie_plaszczyzna} Let $\mi$ be an isotropic log--concave measure in $\R^n$, let $g$ be its density and let $H$ be any hyperplane. Then $g$ restricted to $H$ is a log--concave function, and there exist such constants $c, C$ that $c \leq \int_H g \leq C$ and for any $\theta \in H$, $\|\theta\|_2 = 1$ we have $c \leq \int_H \is{x}{\theta}^2 g(x) dx \leq C$. \end{fact} 

\begin{proof} For the first part, project $\mi$ onto $H^\perp$. We obtain a 1--dimensional isotropic log--concave measure $\nu_1$, and the integral $\int_H d\mi$ is equal to the density of $\nu_1$ at zero, which is bounded from above and below by Fact \ref{szacowanie_gestosci}. 

For the second fact, we project $\mi$ onto the plane spanned by $H^\perp$ and $\theta$, the result is a 2--dimensional isotropic log--concave measure $\nu_2$. Let $\LL$ be the line spanned by $\theta$ and let $h$ be the density of $\nu_2$. We have $C > \int_{\LL} h(x) dx > c$ and $C > \sup h(x) \geq h(0) > c$, thus $\int_{\LL} x^2 h(x)$ also has to be bounded from above and below, as $\tilde{h}(x) = \frac{dh(xd)}{\int_{\LL} h(y) dy}$ is log--concave and isotropic, where $d = \sqrt{\int_{\LL} x^2h(x) dx}$, and from Fact \ref{szacowanie_gestosci} we obtain that $\sup \tilde{h} = \frac{d \sup h}{\int_{\LL} h(y) dy}$ has to be bounded from below and above.\end{proof}

\subsection{Symmetrizations}

\begin{Def} Let $f$ be any bounded measureable function on $\R$. Then the {\em symmetrization} of $f$ is the unique symmetric upper--continuous function $\tf$ which is decreasing on $\R_+$ and satisfies $\lambda\{x : \tf(x) \geq c\} = \lambda\{x : f(x) \geq c\}$ for any $c$. If $f$ is a bounded measureable function on $\R^n$, and $\theta$ is a direction in $\R^n$, then the {\em symmetrization} of $f$ in the direction $\theta$ is such a function $\tf$ that for any $v \in \theta^\perp$ the function $\tf$ restricted to $v + \theta \R$ is the symmetrization of $f$ restricted to $v + \theta \R$.\end{Def}

\begin{prop} \label{logsym}The symmetrization of a log--concave function is log--concave.\end{prop}

\begin{proof} Obviously the symmetrization of the logarithm is the logarithm of the symmetrization, thus it is enough to check that the symmetrization of a concave function is concave. Take any two points $x + r\theta$ and $y + s\theta$ with $x,y \in \theta^\perp$, let $\tf(x + r\theta) = a$ and $\tf(y + s\theta) = b$. This means that $\lambda \{u : f(x + u\theta) \geq a\} = \lambda\{u : \tf(x + u\theta) \geq a\} \geq 2r$. As $f$ is concave, $I := \{u : f(x + u\theta) \geq a\}$ is an interval of length at least $2r$, similarly $J := \{u : f(y + u\theta) \geq b\}$ is an interval of length at least $2s$. Take any $t \in [0,1]$, then for $v \in tI + (1-t)J$ we have (by the concavity of $f$) $f(v) \geq ta + (1-t)b$. Meanwhile $tI + (1-t)J$ is an interval of $tx + (1-t)y + \theta \R$ of length at least $2(tr + (1-t)s)$. Thus $\lambda \{u : f(tx + (1-t)y + u\theta) \geq ta + (1-t)b\} \geq 2(tr + (1-t)s)$, and so $\tf(t(x + r\theta) + (1-t)(y+ s\theta)) \geq ta + (1-t)b$, which (as $x,y,r,s$ and $t$ were arbitrary) proves the concavity of $\tf$.\end{proof}

\begin{prop} \label{symmet} Let $\mi$ be a log--concave measure in $\R^n$ with density $g$ and a diagonal covariance matrix. Let $\tilde{g}$ be the symmetrization of $g$ in one of the coordinate directions $e_i$, and $\tilde{\mi}$ be the measure with density $\tilde{g}$. Then $\tilde{\mi}$ also has a diagonal covariance matrix, with $\E_\mi\is{X}{e_j}^2 = \E_{\tilde{\mi}}\is{X}{e_j}^2$ for $j \neq i$. Moreover there exists such a constant $c(n) > 0$ dependent only on dimension that $\E_\mi \is{X}{e_i}^2 \geq \E_{\tilde{\mi}} \is{X}{e_i}^2 \geq c(n) \E_\mi \is{X}{e_i}^2$.\end{prop}

\begin{proof} First notice that the joint distribution of all $\is{X}{e_j}$ for $j \neq i$ is the same for $X$ distributed according to $\mi$ and to $\tilde{\mi}$, as the projections of these two measures onto $\spann \{e_j, j\neq i\}$ is the same --- thus all but the $i$--th row and column of the respective covariance matrices are the same. Furthermore as $\tilde{g}$ is symmetric with respect to the hyperplane $\spann \{e_j, j\neq i\}$, we have $\E_{\tilde{\mi}} \is{X}{e_i}\is{X}{e_j} = 0$ for $i \neq j$. Also $\E_{\tilde{\mi}} \is{X}{e_i}^2 = \int_{e_i^{\perp}} \int_\R x^2 \tilde{g}(v + xe_i) dx dv \leq \int_{e_i^{\perp}} \int_\R x^2 g(v + xe_i) dx dv$ by the monotone rearrangement inequality.

Consider the diagonal map $T_\mi$ (resp. $T_{\tilde{\mi}}$) transforming the measure $\mi$ to an isotropic measure. The $i$--th entry on the diagonal of the matrix of $T_\mi$ is equal to $1\slash \sqrt{a_{ii}}$, where $a_{ii}$ (resp. $\tilde{a}_{ii}$) is the $i$--th entry on the diagonal of the covariance matrix of $\mi$. Let $M_{\mi}$ (resp. $M_{\tilde{\mi}}$) denote the supremum of the density of $\mi \circ T_\mi^{-1}$ (resp. $\tilde{\mi} \circ T_{\tilde{\mi}}^{-1}$), and let $M$ denote the common supremum of the densities of $\mi$ and $\tilde{\mi}$. We have $M_\mi = M \slash \det T_\mi$, and $M_{\tilde{\mi}} = M \slash \det T_{\tilde{\mi}}$, so 
$$M_{\mi} \slash M_{\tilde{\mi}} = \det T_\mi \slash \det T_{\tilde{\mi}} = \sqrt{\tilde{a}_{ii}} \slash \sqrt{a_{ii}} = \sqrt{\E_{\tilde{\mi}} \is{X}{e_i}^2 \slash \E_{\mi} \is{X}{e_i}^2},$$ where the middle equality follows as all eigenvalues of $\cov \mi$ and $\cov \tilde{\mi}$ except the $i$--th are equal.

On the other hand both $\mi \circ T_\mi$ and $\tilde{\mi} \circ T_{\tilde{\mi}}$ are isotropic log--concave measures, thus by Fact \ref{szacowanie_gestosci} we have $M_{\mi} \slash M_{\tilde{\mi}} \geq c(n) \slash C(n)$. This ends the proof.\end{proof}

\begin{prop} \label{supremumbounds}Let $\mi$ be a probabilistic log--concave measure with density $g$ and a diagonal covariance matrix. For $z \in e_i^\perp$ let $g'(z) = \sup_{x \in \R} g(z + re_i)$. Then:
\begin{itemize}
\item $g'(z)$ is a log--concave function on $e_i^\perp$;
\item $c(n) \slash \sqrt{\E \is{X}{e_i}^2} \leq \int_{e_i^\perp} g'(z) \leq C(n) \slash \sqrt{\E \is{X}{e_i}^2}$.
%\item $c(n) \E \is{X}{e_j}^2 \slash \sqrt{\E \is{X}{e_i}^2} \leq \int_{e_i^\perp} g'(x) \is{x}{e_j}^2 dx \leq C(n) \E \is{X}{e_j}^2 \slash \sqrt{\E \is{X}{e_i}^2}$ for $j \neq i$.
\end{itemize}
\end{prop}

\begin{proof}
Let $\tilde{g}$ be the symmetrization of $g$ in direction $e_i$. Then $\tilde{g}(z,0) = g'(z)$ --- both are equal to $\sup_{x \in \R} g(z + re_i)$. Thus by Proposition \ref{logsym} $g'$ is log--concave. By Proposition \ref{symmet} the measure with density $\tilde{g}$ has a diagonal covariance matrix, as before consider a diagonal map $T$ which transforms this measure into an isotropic one, let $h$ be the density of this isotropic measure. By Fact \ref{ciecie_plaszczyzna} we have that $c(n) \int_{e_i^\perp} h \leq C(n)$. Now the $i$--th entry on the diagonal of $T$ is $\sqrt{\E \is{X}{e_i}^2}$, thus this is the factor by which the mass on the hyperplane $e_i^\perp$ is changed by $T$. %A similar argument gives the third part of the thesis.
\end{proof}

\subsection{The Lipschitz invariant}

\begin{prop} \label{addsections} Let $X_1$ and $X_2$ be two independent random variables in $\R^n$ satisfying $$ \forall_{\theta : \|\theta\|_2 = 1} \ c < \E \is{X_i}{\theta}^2 < C$$ for some positive constants $c,C$. Let $W$ and $V$ be two diagonal matrices with $W^2 + V^2 = \Id$. Then $c < \E \is{W X_1 + V X_2}{\theta}^2 < C$ \end{prop}

\begin{proof}
Let $G_1$ and $G_2$ denote the covariance matrices of $X_1$ and $X_2$, respectively. The assumption upon $\E \is{X_i}{\theta}^2$ means simply that all the eigenvalues of $G_1$ and $G_2$ lie in the interval $[c,C]$.

We have $\Cov (W X_1 + V X_2) = \Cov (W X_1) + \Cov (V X_2) = W G_1 W + V G_2 V$.
Take any vector $\theta$ of norm $1$. Then $\is{W G_1 W \theta}{\theta} = \is{G_1 W \theta}{W\theta}$. As all the eigenvalues of $G_1$ are no smaller than $c$, we have $\is{G_1 v}{v} \geq c \|v\|^2$ for any $v$, thus $\is{W G_1 W \theta}{\theta} \geq c \|W\theta\|^2$. Similarly $\is{V G_2 V \theta}{\theta} \geq c \|V\theta\|^2$. As $V$ and $W$ are diagonal, we have $\|W \theta\|^2 = \sum w_{ii}^2 \theta_i^2$, and as $W^2 + V^2 = \Id$, we have $\|W \theta\|^2 + \|V\theta\|^2 = \sum (w_{ii}^2 + v_{ii}^2) \theta_i^2 = \sum \theta_i^2 = \|\theta\|^2 = 1$. Thus $$\|\Cov (W X_1 + V X_2) \theta\| \geq \is{\Cov (W X_1 + V X_2) \theta}{\theta} \geq c,$$ so in particular all the eigenvalues of $\Cov (W X_1 + V X_2)$ must be greater or equal $c$ (they are all positive reals, as eigenvalues of a covariance matrix). A similar argument gives the upper bound on the eigenvalues.\end{proof}

\begin{lem} \label{mainlemma} Let $g$ and $h$ be the densities of two independent log--concave random variables $X$ and $Y$ in $\R^n$ with diagonal covariance matrices satisfying $\Cov (X+Y) = \Id$. Let  $f'(v) = \sup_{t \in \R} f(v + te_i)$ for $f \in \{g,h\}$ and $v \in e_i^{\perp}$. Then $$\int_{e_i^\perp} g'(v) h'(z-v) dv \leq C(n) \slash \sqrt{\E \is{X}{e_i}^2 \E \is{Y}{e_i}^2}$$ for some constant $C(n)$ dependent only on $n$.\end{lem}

\begin{proof}
Begin by considering the symmetrization $g_1$ of $g$ with respect to $e_i$. Let this be the density of the variable $\tilde{X}$. By Lemma \ref{symmet} $\tilde{X}$ has a diagonal covariance matrix $G$, and $G^{-1\slash 2} \tilde{X}$ is isotropic.

Let $g_2$ be the density of $G^{-1\slash 2} \tilde{X}$. Consider the restriction  of $g_2$ to $e_i^{\perp}$. Let $c_2 :=\int_{e_i^\perp} g_2$ and let $g_3 := g_2 \slash c_2$, $g_3$ is a density of a probability measure on $e_i^\perp$, it is also obviously log--concave. Let $X_3$ be the random variable distributed according to $g_3$. By Fact \ref{ciecie_plaszczyzna} the eigenvalues of $\Cov X_3$ are in some universal interval $[c, C]$ and $c < c_2 < C$.

Let $g_4$ be the restriction of $g_1$ to $e_i^{\perp}$, and $c_4 := \int_{e_i^\perp} g_4$ and $g_5 := g_4 \slash c_4$, finally let $X_5$ be the random variable distributed according to $g_5$. Note that $g'$ is equal to $g_4$.

By the construction above and by Lemma \ref{symmet} we have \begin{equation}\label{e1}c_2 \slash c_4 = \sqrt{G_{ii}} \geq c(n) \sqrt{(\Cov X)_{ii}}.\end{equation} Also let $G'$ be the restriction of $G$ onto $e_i^{\perp}$ (that is, the matrix obtained from $G$ by deleting the $i$--th row and column). Then $X_5 = G'^{1\slash 2} X_3$. We perform similar operations on $h$, to receive \begin{equation}\label{e2}d_2 \slash d_4 \geq c(n) \sqrt{(\cov Y)_{ii}},\end{equation} $Y_5 = H'^{1\slash 2} Y_3$ and the eigenvalues of $\Cov Y_3$ are in the same interval $[c, C]$. 

We are now in the situation of Proposition \ref{addsections}. We have random variables $X_3$ and $Y_3$, satisfying $c < \E \is{X_3}{\theta}^2 < C$ for $\|\theta\| = 1$, and the same for $Y_3$, and matrices $G'^{1\slash 2}$ and $H'^{1\slash 2}$, whose squares sum up to the identity matrix by Proposition \ref{symmet}. Thus the variable $X_5 + Y_5$ satisfies the conclusion of Proposition \ref{addsections}, that is all the eigenvalues of its covariance matrix lie in the interval $[c, C]$.

The integral we consider, $\int_{e_i^\perp} g'(v) h'(z-v)$, is equal to $\frac{d_4c_4}{d_2c_2}$ times the density of the variable $X_5 + Y_5$ at $z$. Thus by (\ref{e1}) and (\ref{e2}) we only have to prove the density of $X_5 + Y_5$ is bounded by a constant dependent on $n$.

Let $M$ be such a diagonal matrix that $M (X_5 + Y_5)$ is isotropic, that is $M = \Cov(X_5+Y_5)^{-1\slash 2}$. Then the density of $M(X_5 + Y_5)$ is bounded from above by $C(n)$ by Fact \ref{szacowanie_gestosci}, and the supremum of the density of $X_5 + Y_5$ is equal to the supremum of the density of $M(X_5 + Y_5)$ multiplied by $\det M = \det \Cov(X_5+Y_5)^{-1\slash 2} \leq c^{-n\slash 2}$.
\end{proof}

And now for the final result of this section:

\begin{Def}
We say a function $f :\R^n \to \R$ is {\em Lipschitz in direction $\theta$} (where we assume $\|\theta\|_2 = 1$) with the constant $L$ if for any $x \in \R^n$ and $t \in \R$ we have $|f(x+t\theta) - f(x)| \leq L|t|$.
\end{Def}

\begin{prop} \label{mainprop} Let $X$ and $Y$ be two such independent log--concave random variables in $\R^n$ with diagonal covariance matrices that $X+Y$ is isotropic. Then if $\E \is{X}{e_i}^2$ and $\E \is{Y}{e_i}^2$ are both positive, then the density of $X+Y$ is Lipschitz in direction $e_i$ with the Lipschitz constant bounded by $C(n) \slash \sqrt{\E \is{X}{e_i}^2 \E \is{Y}{e_i}^2}$\end{prop}

\begin{proof}
We may assume $X$ and $Y$ have densities $g$ and $h$ respectively by convoluting each with a gaussian variable $\eps G$ (where $G \sim \mathcal{N}(0,\Id)$), when $\eps \ra 0$, the density of $X + Y + \eps G + \eps G$ tends to the density of $X+Y$, and the Lipschitz constant is preserved under pointwise convergence.

The density $f$ of $X + Y$ is the convolution of $g$ and $h$. For any $v \in \R^n$ we denote its decomposition into $e_i^\perp$ and $\spann\{e_i\}$ by $v = \tilde{v} + t_v e_i$. We have 
$$f(x) = \int_{\R^n} g(y) h(x-y) dy = \int_{e_i^\perp} \int_{\R} g(\tilde{y} + t_y e_i) h(\tilde{x} - \tilde{y} + (t_x - t_y) e_i) dt_y d\tilde{y},$$
and so the difference $|f(x+se_i) - f(x)|$ is equal to 
$$\Bigg|\int_{e_i^\perp} \int_{\R} g(\tilde{y} + t e_i) \Big[h(\tilde{x} - \tilde{y} + (t_x + s - t_y) e_i) - h(\tilde{x} - \tilde{y} + (t_x - t_y) e_i)\Big] dt_y d\tilde{y}\Bigg|.$$
Let $g'(\tilde{v}) = \sup_{t\in\R} g(\tilde{v}+te_i)$ and $h'(\tilde{v}) = \sup_{t\in\R}(\tilde{v}+te_i)$. Then we have
\begin{align*} |f(x+se_i) - f(x)| \leq& \int_{e_i^\perp} \int_{\R} g(\tilde{y} + t e_i)  \Big|h(\tilde{x} - \tilde{y} + (t_x + s - t_y) e_i)  \\  & - h(\tilde{x} - \tilde{y} + (t_x - t_y) e_i)\Big| dt_y d\tilde{y} \\
 \leq& \int_{e_i^\perp} g'(\tilde{y}) \int_{\R} \Big|h(\tilde{x} - \tilde{y} + (t_x + s - t_y) e_i) \\  & - h(\tilde{x} - \tilde{y} + (t_x - t_y) e_i)\Big| dt_y d\tilde{y}.\end{align*}
Notice that $h(\tilde{v} + te_i)$ is log--concave, and thus bimonotonous, as a function of $t$. Thus the function $h(\tilde{v} + (t+s)e_i) - h(\tilde{v} + te_i)$ for positive $s$ has a single zero at some $t_0$, is non--negative for $t < t_0$ and non--positive for $t > t_0$. Thus 
\begin{align*}\int_\R & |h(\tilde{v} + (t+s)e_i) - h(\tilde{v} + te_i)| dt  \leq \int_{-\infty}^{t_0}h(\tilde{v} + (t+s)e_i) - h(\tilde{v} + te_i) dt \\ &- \int_{t_0}^\infty h(\tilde{v} + (t+s)e_i) - h(\tilde{v} + te_i) dt
= 2 \int_{t_0}^{t_0+s} h(\tilde{v} + te_i) \leq 2|s|h'(\tilde{v}).
\end{align*}
Applying this to our case we have
\begin{align*} |f(x+se_i) - f(x)| &\leq \int_{e_i^\perp} g'(\tilde{y}) 2 |s| h'(\tilde{x} - \tilde{y}) d\tilde{y}.\end{align*}

Here, however, we may apply Lemma \ref{mainlemma} to conclude the proof.\end{proof}

\subsection{Closed classes}

We begin by demonstrating a simple structural proposition:

\begin{prop} \label{structure} Let $\KK$ be a class of log--concave measures. Let $\bKK$ be the smallest class of log--concave measures containing $\KK$, which is closed under products, linear transformations and weak limits. Let $\mi$ be any isotropic measure in $\bKK$. Then there exists a sequence of measures $\mi_1,\mi_2,\ldots$, each of which is a linear image of the tensor product of finitely many of measures from $\KK$, such that $\mi$ is the weak limit of $\mi_i$, and each $\mi_i$ is isotropic.\end{prop}

\begin{proof}
First let $\LL$ be the smallest class of log--concave measures containing $\KK$, which is closed under products and linear transformations. We shall prove any member of $\bKK$ is a weak limit of some sequence in $\LL$. Let $\bLL$ be the set of all weak limits of sequences in $\LL$. We will prove $\bLL$ is closed under products, linear transformations and weak limits, thus $\bLL = \bKK$. Let $\mi_1,\mi_2 \in \bLL$, let $\mi_1 = \lim \nu_n$, $\mi_2 = \lim \omega_n$. Then $\mi_1 \otimes \mi_2 = \lim \nu_n \otimes \omega_n$. %Take any continuous $f$. Then $\E f(X_n, Y_n) = \E_X \E_Y f(X_n, Y_n)$. For fixed $x$ the function $f(x, y)$ is continuous in $y$, thus $\E_X \E_Y f(X_n, Y_n) = \E_X \E_Y f(X_n, Y) = \E_Y \E_X f(X_n, Y) =$ (similarly) = $\E_Y \E_X f(X,Y) = \E f(X,Y)$.
Similarly, if $T$ is a linear transformation, then $(\lim \nu_n) \circ T^{-1} = \lim (\nu_n \circ T^{-1})$. %Take any continuous $f$. Then $\E f(T(X_n)) = \E f(T(X))$ for $T$ is continuous, being linear.
Thus $\bLL$ is closed under products and linear transformations. Now recall that the space of probabilistic measures with the weak convergence can be given as a metric space with the Levy metric. Thus the closure of any set with respect to weak limits is simply the closure in the Levy metric, and this is precisely the set of all weak limits of sequences in this set. Thus $\bLL = \bKK$.
%For if we have a sequence $\mi_n \in \bLL$ convergent to $\mi$, and $\mi_n = \lim \nu_i^n$, then $\mi$ is also the limit of an appropriate sequence $\nu_{i(n)}^n$.

Let $\mi = \lim \mi_k$, $\mi_k \in \LL$. Let $\nu_k = \mi_k \circ (\cov \mi_k)^{-1\slash 2}$. As $\mi$ is isotropic and the covariance matrix is continuous for log--concave measures, $\cov \mi_k \ra \Id$ and thus also $\nu_k \ra \mi$. The measures $\nu_k$ are isotropic.

Now we have to consider $\nu_k$, which are elements of $\LL$. We prove each $\nu_k$ is the linear image of a product by structural induction, exchanging all linear transformations with products, as $(\mi \circ S^{-1}) \otimes (\nu \circ T^{-1}) = (\mu \otimes \nu) \circ (S \otimes T)^{-1}$. \end{proof}

\begin{lem} \label{liplimit} Let $\mi_n$ be a sequence of probabilistic isotropic log--concave measures with densities $f_n$, weakly convergent to some $\mi$. Assume each $f_n$ is Lipschitz in some direction $v_n$ with the same constant $L$. Then $\mi$ has a density $f$ that is Lipschitz with the constant $L$ in some direction $v$.\end{lem}

\begin{proof} First note that $\mi$ has to be log--concave and isotropic, and thus has a density. Now pass to such a subsequence of $n$s that $v_n$ is convergent to some $v$.

Let $h_{\delta,z}(x) = \max\{1 - |x-z|\slash \delta, 0\}$, and let $T_{\delta, z}(\mi) = \int h_{\delta,z}(x) d\mi(x) \slash \int h_{\delta,z}(x) dx$. For any $n$ the function $T_{\delta,z}(\mi_n)$ is $L$--Lipschitz in direction $v_n$ as a function of $z$. Consider any fixed point $z$ and any constant $t > 0$. We have 
$$|T_{\delta,z + tv_n}(\mi_n) - T_{\delta,z + tv}(\mi)| \leq |T_{\delta,z + tv_n}(\mi_n) - T_{\delta,z+tv}(\mi_n)| + |T_{\delta,z+tv}(\mi_n) - T_{\delta,z + tv}(\mi)|.$$ The first part tends to zero as $h_{\delta,z+tv_n}$ tends to $h_{\delta,z + tv}$ uniformly, while the density of $\mi_n$ is bounded uniformly in $n$. The second part tends to zero as $T$ is a continuous functional of a measure. Thus, in particular, $T_{\delta, z}(\mi)$ is a $L$--Lipschitz function of $z$. Note that the Lipschitz constant is independent of $\delta$.

We modify $f$ to be zero on the boundary of $\supp \mi$ --- this is a modification on a set of measure 0, so the modified $f$ is also a density of $\mi$. We will prove $f$ is now $L$--Lipschitz in the direction $v$. Take any point $z$ and consider the line $\LL = z + tv$. If this line does not intersect the interior of $\supp \mi$, $f$ is equal to 0 on $\LL$, and thus is $L$--Lipschitz. Now suppose $\LL$ intersects the interior of $\supp \mi$. As $\supp \mi$ is convex, $\LL$ intersects the boundary of $\supp \mi$ in exactly two points. Take any two points $x,y$ on $\LL$, different from the two intersection points. Then $f$ is continuous in some neighbourhoods of $x$ and $y$, for $f$ is continuous both in the interior of $\supp \mi$ and outside $\supp \mi$. Thus $T_{\delta,x} \ra f(x)$ and $T_{\delta,y} \ra f(y)$ when $\delta \ra 0$, so $f$ is Lipschitz everywhere except the two boundary points (but also when $x,y$ straddle a boundary point). To deal with the case of $x$ or $y$ lying on the boundary of $\supp \mi$ we take a sequence $x_n \ra x$ (or $y_n \ra y$, respectively).\end{proof}

\begin{lem} \label{twocases} Let $(d+1)^{-2} > \eps > 0$. Let $X_1,X_2,\ldots,X_n$ be such a sequence of independent random variables in $\R^d$ that $\sum_{i=1}^n X_i$ is isotropic. Let $\mathcal{M}_i$ denote the space spanned by the eigenvectors of $\cov X_i$ corresponding to eigenvalues larger than $1-\eps$. Then:\begin{itemize}
\item either $\sum_i \dim \mathcal{M}_i = d$,
\item or there exists a subset $S$ of $\{1,2,\ldots,n\}$ such that $\cov \sum_{i\in S} X_i$ has an eigenvalue $\lambda$ satisfying $\eps < \lambda < 1-\eps$.
\end{itemize}\end{lem}

\begin{proof} Let $A_i$ denote the covariance matrix of $X_i$. If any single $A_i$ has an eigenvalue between $\eps$ and $1-\eps$, we set $S = \{i\}$. Consider eigenvalues of $A_i$s larger than $1 - \eps$, assume $A_i$ has $k_i$ such eigenvalues, let $k = \sum_i k_i$. We have $k \leq d$, for otherwise $d = \tr \id = \tr \sum A_i \geq \sum_i k_i (1-\eps) \geq (d+1)(1-\eps) > d$. If $k = d$ the first condition is satisfied. Otherwise there is at most $d-1$ indices $i$ with $k_i > 0$, and the sum of the traces of the appropriate $A_i$s is at most the sum of the large eigenvalues (at most $d-1$) and the sum of the small eigenvalues (at most $(d-1)^2 \eps$, for each of the matrices has at most $d-1$ small eigenvalues). Thus the trace of the sum of the remaining matrices is at least $1 - (d-1)^2 \eps$.

Rearrange the vectors so $k_i = 0$ for $k = 1,2,\ldots,l$, and $k_i > 0$ for $i > l$. Consider the sums $b(m) = \sum_{i=1}^m \tr A_i$. We have $b(0) = 0$, $b(l) \geq 1 - (d-1)^2 \eps \geq 4d \eps$ and $b(m+1) - b(m) \leq d\eps$. Thus for some $m_0$ we shall have $d\eps \leq b(m_0) \leq 2d\eps$. We put $S = \{1,2,\ldots,m_0\}$. The vector $\sum_{i=1}^{m_0} X_i$ has to have an eigenvalue of the covariance matrix no smaller than $\eps$ (for the trace of this matrix is at least $d\eps$), on the other hand all its eigenvalues are no larger than the trace, which in turn is no larger than $2d\eps < 1 - \eps$. Thus $S$ satisfies the theorem.
\end{proof}

\begin{thm}\label{nomulti} Let $Y_k = \sum_{i=1}^{i(k)} X_{k,i}$ be isotropic, log--concave random variables on $\R^n$. We assume all $X_{k,i}$ are independent. Assume $Y_k \ra Y$ weakly, where $Y$ is a uniform measure on a convex body. Then $Y = Z_1 + Z_2 + \ldots Z_m$ for some independent random variables $Z_j$, where each $Z_j$ is a weak limit of some of the $X_{k,i}$s, and all $Z_j$ are supported on orthogonal subspaces.\end{thm}

\begin{proof}
Choose any $\eps > 0$ and apply Lemma \ref{twocases} for each $Y_k$. If for some $k$ the second case of the lemma occurs, we have $Y_k = Y_k' + Y_k''$ (where $Y_k'$ is the sum of the $X_{k,i}$s for $i \in S$, and $Y_k''$ is the sum of the others), and some eigenvalue of $Y_k'$ is between $\eps$ and $1-\eps$. As $Y_k$ is isotropic, $Y_k''$ has the same eigenvectors as $Y_k'$, thus we may apply Proposition \ref{mainprop} to obtain that the density of $Y_k$ is $C(d)\slash \eps$--Lipschitz in the direction of the appropriate eigenvector. If this occurs an infinite number of times, we may apply Lemma \ref{liplimit} to the subsequence of those $k$s to obtain that the density of $Y$ is $C(d)\slash \eps$--Lipschitz in some direction, and this contradicts the assumption that $Y$ was uniform on a convex body.

Thus for any $\eps > 0$ the second case of the lemma occurs only finitely many times. We may thus pass to a subsequence on which the first case occurs for each $Y_k$ with some $\eps_k$ tending to zero. Thus for each each $k$ we have a set of at most $d$ linear spaces $\mathcal{M}_{k,i}$. We pass to a subsequence again, so that the number and dimensions of the subspaces are constant, and again (after an appropriate rearrangement) to have $\mathcal{M}_{k,i} \ra \mathcal{M}_i$ (the convergence of linear subspaces is taken, for instance, in the metric of the grassmanian manifold, this can be done due to the compactness of this manifold).

The spaces $\mathcal{M}_i$ have to be orthogonal. Assume, say, $\mathcal{M}_1$ and $\mathcal{M}_2$ are not orthogonal, say some two unit vectors $v_1 \in \mathcal{M}_1$ and $v_2 \in \mathcal{M}_2$ satisfy $\is{v_1}{v_2} = c < 0$. Then we have sequences $v_{k,1} \ra v_1$ in $\mathcal{M}_1$ and $v_{k,2} \ra v_2$ in $\mathcal{M}_2$. Note the following inequalities:
\begin{align*}1 &= \E \is{v_{k,1}}{Y_k}^2 \geq \E\is{v_{k,1}}{X_{k,1} + X_{k,2}}^2 = \E \is{v_{k,1}}{X_{k,1}}^2 + \is{v_{k,1}}{X_{k,2}}^2 \\ &\geq 1 - \eps_k + \is{v_{k,1}}{X_{k,2}}^2,\end{align*} thus $\E \is{v_{k,1}}{X_{k,2}}^2 \leq \eps_k$. Moreover 
\begin{align*}\E \is{v_{k,1} + v_{k,2}}{X_{k,1}}^2 &= \E \is {v_{k,1}}{X_{k,1}}^2 + 2\is{v_{k,1}}{X_{k,1}} \is{v_{k,2}}{X_{k,1}} + \is{v_{k,2}}{X_{k,1}}^2 \\ & \geq 1 - \eps_k - 2|\E \is{v_{k,1}}{X_{k,1}} \is{v_{k,2}}{X_{k,1}}| \geq 1 - \eps_k - 2\sqrt{\eps_k}.\end{align*} Finally $$\E \is{v_{k,1}+v_{k,2}}{Y_k}^2 \geq \E \is{v_{k,1} + v_{k,2}}{X_{k,1} + X_{k,2}}^2 \geq 2(1 - \eps_k - 2\sqrt{\eps_k}),$$ which is arbitrarily close to $2$ for large enough $k$. On the other hand, however, $\is{v_{k,1} + v_{k,2}}{v_{k,1}+v_{k,2}} = 2 + 2\is{v_{k,1}}{v_{k,2}} \ra 2 - 2c$, thus we should have $\E \is{v_{k,1}+v_{k,2}}{Y_k}^2 \ra 2 - 2c$, a contradiction.

Now each $X_{k,i}$ for fixed $i$ converges weakly to some measure $Z_i$ distributed on $\mathcal{M}_i$ (for the variance in the directions orthogonal to $\mathcal{M}_i$ tends to zero, as shown above), and $Y$ is the sum of $Z_i$s.
\end{proof}

Now all that remains is to combine the theorem above with Proposition \ref{structure}:

\begin{thm}\label{nomulti2}
Let $\KK$ be any class of log--concave measures closed under linear transformations, and let $\bKK$ be the smallest class of log--concave measures containing $\KK$ which is closed under product, weak limits and linear transformations. Let $\mi$ be any isotropic measure in $\bKK$ which is a uniform measure on some convex body. Then $\mi$ is a product of some measures $\mi_1, \mi_2,\ldots, \mi_n$, each of which is a weak limit of some sequence of linear images of measures in $\KK$.
\end{thm}

\subsection{Applications and discussion}
Let us apply this result to a typical case. 
\begin{cor} Let $\KK$ be the smallest class of log--concave measures, closed under products, linear transformations and weak limits, which contains all $1$--dimensional log--concave measures. Let $\mi \in \KK$ be the uniform measure distributed on a convex body $B$. Then $B$ is a parallelotope.\end{cor}

\begin{proof} The class of 1--dimensional log--concave measures is closed under linear transformations and weak limits, thus by Theorem \ref{nomulti2} any uniform measure on a convex body in $\KK$ which is isotropic has to be a product of 1--dimensional log--concave measures, and thus equidistributed on the hypercube. Taking any, not necessarily isotropic, measure equidistributed on some $B$ in $\KK$ we can take a linear transformation to make it isotropic, and as after this transformation $B$ is a hypercube, it had to be a parallelotope before.
\end{proof}

Of course a case--by--case survey of what can be obtained from various classes of log--concave measures is impossible here. However the main scheme should be clear: we begin with some class of measures $\KK$, and if there are no non--trivial convex bodies obtained as weak limits of linear transformations of measures from $\KK$, then $\bKK$ contains no non--trivial convex bodies. It would be interesting to prove, for instance, that the class of $\ell_p$ balls (for fixed $p$ between 1 and $\infty$) generates no non--trivial convex bodies (that is, convex bodies not being linear transforms of $\ell_p$ balls).

\end{document}